\documentclass[twoside,bezier,12pt, reqno]{amsart}
\usepackage{amsmath, amsthm, amscd, amsfonts,url, amssymb,graphicx,graphics, color}
\usepackage[bookmarksnumbered, plainpages]{hyperref}
\usepackage{pgf,tikz}
\usetikzlibrary{arrows}
\newtheorem{theorem}{Theorem}[section]
\newtheorem{lemma}[theorem]{Lemma}

\newtheorem{corollary}[theorem]{Corollary}
\theoremstyle{definition}
\newtheorem{definition}[theorem]{Definition}

\theoremstyle{remark}
\newtheorem{remark}[theorem]{Remark}
\numberwithin{equation}{section}
\begin{document}

\title[]{more odd graph theory from another point of view}%
\author{S. Morteza  MIRAFZAL }%
\address{Department of Mathematics, Lorestan University, Khoramabad, Iran}%
\email{smortezamirafzal@yahoo.com}
\email{mirafzal.m@lu.ac.ir}%
\address{ }%
\email{}%


\begin{abstract} The Kneser graph $K(n, k)$ has as vertices all $k$-element subsets of $[n]=\{1,2,...,n \}$
and an edge between any two vertices  that are disjoint. If
$n=2k+1$, then  $K(n, k)$ is called an odd graph. Let $ n >4$ and $1< k < \frac{n}{2} $.
In the present  paper,
we show that if the Kneser graph $K(n,k)$  is   of even order where  $n$ is an  odd integer or both of the integers $n,k $
 are even,
then $K(n,k)$   is a vertex-transitive non Cayley graph.
Although,   these are  special cases of   Godsil [8], unlike his proof that
 uses some very  deep group-theoretical facts,  ours uses no heavy group-theoretic facts.
We   obtain our results by using  some rather
 elementary facts of
 number theory and group theory.
We show that `almost all' odd graphs are of even order, and consequently are
 vertex-transitive non Cayley graphs.
 Finally, we show that if $ k > 4 $ is an
  even integer  such that  $k$ is not of the form $ k = 2^t  $  for some $ t> 2 $,  then
the line graph of the odd graph $O_{k+1}$ is a vertex-transitive non Cayley graph.

 \


\end{abstract}  \

\maketitle
\section{Introduction and Preliminaries} In this paper, a graph
$\Gamma=(V,E)$ is considered as a finite, undirected, connected graph, without loops or multiple edges,  where $V=V(\Gamma)$ is the
vertex-set and $E=E(\Gamma)$ is the edge-set. For every terminology and
notation not defined here, we follow [2, 6, 7, 11].

The study of vertex-transitive graphs has a long and rich history in discrete
mathematics. Prominent examples of vertex-transitive graphs are Cayley graphs
which are important in both theory as well as applications. Vertex-transitive graphs
that are not Cayley graphs, for which we use the abbreviation VTNCG,
have been an object of a systematic study since  1980 [3,8].
In trying to recognize
whether or not a vertex-transitive graph is a Cayley graph, we are left with
the problem of determining whether the automorphism group contains a regular
subgroup [2]. The reference   [1] is an  excellent  source for studying graphs that are VTNCG.\

Let $ n > 4$   be an   integer and $1 < k < \frac{n}{2}$.
The Kneser graph $ K(n,k) $ is  the graph
with the $ k $-element subsets  of $ [n] = \{ 1,2,...,n \} $ as  vertices, where
 two such vertices are adjacent if and only if they are disjoint. If $n=2k+1$, then the
 graph $ K(2k+1,k)$ is called an odd graph and is denoted by $O_{k+1}$.
  There are several good
 reasons for studying these graphs. One is that the questions
which arise are related to problems in other areas of combinatorics, such as
combinatorial set theory, coding theory, and design theory. A second reason is that
the study of odd graphs tends to highlight the strengths and weaknesses of the
techniques currently available in graph theory, and that many interesting problems
and conjectures are encountered  [3,4].  \

Amongst the various interesting properties
of the Kneser graph $ K(n,k)$,   we interested in the automorphism group of it and we
want to see how it acts on its vertex set. If $ \theta \in Sym ([n]) $, \
 then \

\

\centerline{$ f_\theta : V( K(n,k) )\longrightarrow V( K(n,k)) $,
 $ f_\theta (\{x_1, ..., x_k \}) = \{ \theta (x_1), ..., \theta (x_k) \}$} \

 is an automorphism of $ K(n,k) $ and the mapping
  $ \psi : Sym ([n]) \longrightarrow Aut ( K(n,k) )$, defined by
 the rule $ \psi ( \theta ) = f_\theta $ is an injection.
  In fact,  $ Aut ( K(n,k)) = \{ f_\theta  | \theta  \in Sym([n]) \} \cong Sym( [n] )$  [7],
   and for this reason we
 identify $  f_\theta $ with $  \theta$ when  $  f_\theta $
is an automorphism of $ K(n,k) $, and in such a situation we
 write $ \theta $ instead of $  f_\theta $. It is an easy task to
  show that the Kneser graph $ K(n,k)$ is a vertex-transitive graph [7].\

In 1979 Biggs  [3]   asked whether there are many values of $k$
for which the odd graphs $ O_{k+1}$
are Cayley graphs. In 1980 Godsil [8]  proved (by using  some
very deep group-theoretical facts of group theory [9,10]) 
    that for 'almost all' values of $k$, the Kneser graph $ K(n,k)$ is a
$VTNCG$. In the present  paper,
we show that if the Kneser graph $K(n,k)$  is   of even order where  $n$ is an  odd integer or both of the integers $n,k $
 are even,
then $K(n,k)$ is a $VTNCG$.
We call the odd graph $ O_{k+1}$ an even-odd graph when its order is an even integer.
 We show that `almost all' odd graphs are even-odd graphs, and consequently
 'almost all' odd graphs are $VTNCG$. We obtain our results,  by using some rather
elementary  facts of number theory and group theory. \

Finally, we show that if $ k > 4 $ is an
  even integer  and $k$ is not of the form $ k = 2^t  $  for some $ t > 2$,  then
the line graph of the odd graph $O_{k+1}$ is a $VTNCG$.

\section{Main results}

\begin{theorem}
Let $n,k$ are integers, $n>4$, $2 \leq k < \frac{n}{2}$ and  ${n}\choose{k}$ is an even
integer.  Then the Kneser graph $K(n,k)$  is a vertex transitive non Cayley graph
  If one of the following conditions holds: \

 $(I)$ \   \ \  $n$ is an odd integer; \

 $(II)$  \ \  $n$ and $k$ are even integers.
\end{theorem}

\begin{proof}
 We know that  the Kneser graph $ K(n,k) $ is a vertex-transitive graph
 (for every positive integer $n$) [7],
hence it is sufficient to show that it is   a  non Cayley graph.\

On the contrary, we  assume that the Kneser graph $ K(n,k) $ is a Cayley graph.
Then the group $ Aut ( K(n,k)  ) = Sym( [n] ) $, $ [n] = \{ 1, 2, ..., n \} $ has a
 subgroup $ R $, such that $ R $ acts regularly on the vertex-set of $ K(n,k) $.
In particular,  the order of $ R $ is $  { n \choose k } $, and since ( by assumption )  this
number is an even integer, then $ 2 $ divides $| R| $.  Therefore,   by the Cauchy's
theorem the group $ R $ has an element $ \theta $  of order $ 2 $.
We know that each element of $ Sym([n])$ has a unique factorization into disjoint
cycles of $ Sym([n]) $, hence
we can write $ \theta = \rho_1 \rho_2 ... \rho_{h} $, where
each $ \rho_i $ is  a cycle  of $ Sym ( [n] ) $ and
 $ \rho_i \cap \rho_j ={\varnothing} $
when $ i \neq j $. We also know that if $ \theta = \rho_1 \rho_2 ... \rho_{h} $, where
each $ \rho_i $ is  a cycle of $ Sym ( [n] ) $ and
 $ \rho_i \cap \rho_j ={\varnothing} $, then the order
 of the permutation $ \theta $ is the least common multiple
 of the integers,  $ | \rho_1 |, \  |  \rho_2 |,
   ... , | \rho_{h  |} $.     Since $ \theta $ is of order $ 2 $, then the order
of each $ \rho_i $ is $ 2 $ or $ 1 $, say,  $ | \rho_i | \in \{1, 2 \} $.
In   other words,  each $ \rho_i $ is a transposition or a cycle of length $ 1 $. \
Let $ \theta = \tau_1 \tau_2 ... \tau_a ( i_1) (i_2) ... (i_b) $,  where
each $ \tau_r $ is a transposition and each $ i_s \in [n] $. We now argue the  cases (I) and (II). \

(I) \ Let $n=2m+1,\  m >1$.
Therefore, we have $ 2a+b =n=2m + 1 $, where $b$ is an odd integer, and hence it is non-zero.
Since $ b  $ is a positive odd integer, then $ b-1$ is an
 even integer. we let $ d= \frac { b-1} {2}$, so that $ d $ is a non-negative
 integer,  $ d < b$ and $ m= a+d $. Let $ \tau_r = (x_r y_r )$, $ 1 \leq r \leq a$,  where $ x_r, y_r \in [n] $.
 Now,  there are two cases: \

  (i) \  \    $ 2a \leq k$, \  \ (ii) \ \ $ 2a > k$. \

 (i) \ \ Suppose   $ 2a \leq k$. Then there is some integer $t$
 such that $ 2a+ t =k$, and since $ 2a+ b =2m+1 $, then $ t \leq b $.
   Thus,   for transpositions
 $ \tau_1,  \tau_2,  ...,  \tau_{a}$ and cycles $ ( i_1),...,( i_t )$ of
 the cycle factorization
 of $ \theta $, the set  $ v= \{ x_1, y_1,..., x_a, y_a,  i_1,  i_2,  ...,  i_t    \} $
  is a $k $ subset of the set $ [n]$,  and thus it is a vertex
 of the Kneser graph $ K(n,k) $. Therefore,   we have;  \

 $$ \theta( v ) = \{ \theta ( x_1), \theta ( y_1 ), ..., \theta ( x_{a}), \theta (y_{a}), \theta ( i_1 ), ..., \theta ( i_t )\}= $$
 $$ \{ y_1, x_1, ..., y_{a}, x_{a}, i_1,  i_2,  ...,  i_t \} = v$$\

  (ii) \ \ \ Suppose $ 2a > k$. Then there is some integer $c $ such that $ 2c \leq k$ and $ 2(c+1 ) > k$.  If $2c=k$, then we take the vertex $ v= \{ x_1, y_1,..., x_c, y_c\} $, and hence we have;  \

\

 \centerline{$ \theta( v ) = \{ \theta ( x_1), \theta ( y_1 ), ..., \theta ( x_c), \theta (y_c)\} =
  \{ y_1, x_1, ..., y_c, x_c \} = v $}\

  We now assume $2c< k$, then
   $ 2c+1 = k$.  Since $ b \geq 1$, then for transpositions
 $ \tau_1,  \tau_2,  ...,  \tau_{c}$ and cycle $ ( i_1)$ of the cycle factorization
 of $ \theta $, the set  $ v= \{ x_1, y_1,..., x_c, y_c,  i_1\} $ is a $k $-subset of the set $ [n] $,   and therefore it is a vertex
 of the Kneser graph $ K(n,k) $. Thus,   we have;   \newline

 \centerline{$ \theta( v ) = \{ \theta {( x_1)}, \theta( y_1 ), ..., \theta (x_c), \theta (y_c), \theta ( i_1 ) \} =
  \{ y_1, x_1, ..., y_{c}, x_{c}, i_1  \} = v $}\

 (II)
Let $n= 2m, \  m>2$ and $k=2e, \  0< 2e <m$.    Since $ \theta = \tau_1 \tau_2 ... \tau_a ( i_1) (i_2) ... (i_b) $,   \ $ \tau_r = (x_r y_r )$, $ 1 \leq r \leq a$,  where $ x_r, y_r \in [n] $, then  we have $ 2a+b =n=2m  $, where $b$ is an even integer.  We now consider the following  cases. \

If $ a<e $, then $2a<2e$, and hence there is some integer $t$ such that  $2a+t=2e=k$. Since $2a+b=n>2k=4e$, then $t<b$.  Therefore, $v=\{ x_1, y_1,...,x_a,y_a,i_1,...,i_t \}$ is a vertex of $K(n,k)$. \

If $a \geq e$, then   $v=\{ x_1, y_1,...,x_e,y_e \}$ is a vertex of $K(n,k)$. \

On the other hand, we  can see that in every case we have $\theta(v)=v$.\

From the above argument, it follows that $ \theta $ fixes a vertex of the Kneser graph $ K(n,k) $, which is a contradiction, because $R$ acts regularly on the vertex-set of $ K(n,k) $ and $  \theta \in R$ is of order 2.
This contradiction
  shows that the assertion of our theorem is true.

\end{proof}

\begin{remark}
Since $1 < k < \frac{n}{2}$, hence
 if in the case (I) of the  above theorem, we add the condition, \

\centerline{ 'and assume that  $k$  is  an even integer' } \

  then
we can  construct two vertices $v,w$ of $K(n,k)$ such that
 $\theta( v )=v$, $\theta( w )=w$ and $ v\cap w =\varnothing $. In fact,
if $k=2l,\  l>0$, and $2a \leq k=2l$, we have no problem for constructing the vertices $v$ and $w$. If $2a> k=2l$, then
we can construct,
by transpositions, $\tau_1,...,\tau_l $, the vertex $v$, and  hence  we now  have enough transpositions
and fixed points of  $\theta$, namely $\tau_{l+1},..., \tau_{a} $ and  $i_1,...,i_b$, for constructing the vertex $w$. \

In particular, if
 the integers $n,k$ are even, then there are vertices $v,w$ of $K(n,k)$ such that
 $\theta( v )=v$, $\theta( w )=w$ and $ v\cap w =\varnothing $.
\end{remark} \

Since $ {5} \choose {2}$=10 is an even integer and 5 is an odd integer,
 then we can conclude   from  the above theorem  that $ K(5,2)$, namely,   the Petersen graph is
a vertex transitive non-Cayley graph. \

Let $ k$   be a positive integer.
We know from the definition of the odd  graph $O_{k +1} $ that   its order
is $ { 2k+1}  \choose { k } $. \

\begin{definition} We define the odd graph $ O_{k +1} $  as an even-odd graph, when it has
an even order,  namely, if the number  $ { 2k+1}  \choose { k } $ is an even integer.
\end{definition} \

  For   example,
we know that the Petersen graph has $ { 5 } \choose { 2 } $ = 10 vertices,  so    by our
definition  it is an even-odd graph.  In the first step,  we want to show that  `almost all' odd graphs  are even-odd graphs. In fact,
we will show that if $ k > 1 $ and $k$ is not of the form $ k = 2^t -1 $  for some $ t \geq 2$,  then
the number  $ { 2k+1}  \choose { k } $ is an even integer. For our  purpose, we need
the following fact which was discovered by Lucas [5,  chapter   3]. We assume here the usual conventions for binomial coefficients, in particular,  $ {a}\choose {b}$=0  if $ a < b$. \

\begin{theorem}
   Let $ p $ be a prime number  and let $ m= a_0 + a_1 p + ... +a_k p^k$, $ n= b_0 + b_1 p + ... +b_k p^k$,
where $ 0\leq a_i,  b_i < p  $  for $ i= 0,..., k $. Then we have;  \
  \end{theorem}
 \centerline
{ ${m } \choose {n}$  $ \equiv \prod _{ i=0} ^{ k}$ ${a_i} \choose { b_i }$   \ \ (  mod \  $ p$ ) }
   \
\

\

We now are ready to  prove the following assertion.
\begin{lemma}

If $ k > 1 $ and $k$ is not of the form $ k = 2^t -1 $  for some $ t \geq 2$,  then
the number  $ { 2k+1}  \choose { k } $ is an even integer.
\end{lemma}
\begin{proof}
Let $ k= a_0 + a_1 2 + ... +a_n 2^n$, where $ 0\leq a_i  < 2  $ for
 $ i= 0,..., n $ with $ a_n \neq 0 $.
If for every $i$, $ a_i = 1$, then we have $ k = 1+ 2+ 2^2 +...+ 2^n = 2^{ n+1} -1 $,
 which is impossible.
Therefore,  we must have $ a_r =0 $, for some $ r= 0,..., n-1 $. If $ j$ is the largest
index such that $ a_j =0 $, then for all $i$ such that $ i > j$ we have $ a_i =1$.
It is obvious that $ j \neq n$. If we let
$ 2k+1 =b_0 + b_1 2 + ... +b_n 2^n + b_{n+1 }2^{ n+1} $, then  we must
 have $ b_{ i+1} = a_i $ for $ i= 0,..., n $.  Therefore,
 $ {b_{ j+1}} \choose {a_{ j+1}}$ = ${a_{ j}} \choose {a_{ j+1}} $ = $ { 0 } \choose { 1 }$ =0.
 By Theorem 2.4.  we
  have \\

  \centerline{ ${2k+1 } \choose {k}$ $ \equiv \prod _{ i=0} ^{ n +1}$${b_i} \choose {a_i }$ \ \ (  mod \  $ 2\
 )$}

\

 where $ a_{n+1 } =0$.  Now
since $ {b_{ j+1}} \choose {a_{ j+1}}$ = 0, then \\

\centerline{${2k+1 } \choose {k}$ $ \equiv \prod _{ r=0} ^{ n +1}$${b_i} \choose { a_i }$ =0 (  mod \   $ 2\  )$}

\

Thus,  for such a $k$
the integer ${2k+1 } \choose {k}$ is an even integer.

\end{proof} \

\begin{remark}
 It follows from the proof of the above lemma that, if $ k = 2^t -1 $, for some positive integer $t$,
 then the integer ${2k+1 } \choose {k}$ is an odd integer.
 \end{remark}
\

The following result follows from Lemma 2.5.
\begin{corollary}
`Almost all' odd graphs are even-odd graphs, in   other words,
if   $ k > 1 $ and $k$ is not of the form $ k = 2^t -1 $  for some $ t \geq 2$,  then
the odd graph $ O_{k+1} $ is an even-odd graph.

\end{corollary}

Now, it follows from Theorem 2.1.  the following result.

\begin{theorem}
Let $ k$   be a positive integer. If $ O_{k+1} $ is
an even-odd graph, then it is a vertex-transitive non Cayley graph, namely,
$ O_{k+1}$ is a  VTNCG.
\end{theorem} \

Let $ \Gamma $ be a graph, then the line graph $ L( \Gamma )$
of the graph $ \Gamma$ is constructed by taking the edges of $ \Gamma$
as vertices of  $ L( \Gamma )$, and joining two vertices in $ L( \Gamma )$
whenever the corresponding edges in $ \Gamma $ have a common vertex. There is
an important relation between $ Aut( \Gamma)$ and  $ Aut( L(\Gamma))$. In fact, we have
the following result [3, chapter 15]. \
\begin{theorem}
The mapping $ \theta: Aut( \Gamma) \rightarrow Aut( L(\Gamma))$ defined by;
$$ \theta(g) \{ u,v\}= \{  g(u), g(v)\}, \   g \in Aut( \Gamma), \ \{u,v \} \in E( \Gamma ) $$
is a group homomorphism and in fact we have;  \

$(i)$  \  \ \   $ \theta $ is a monomorphism provided  $ \Gamma \neq K_2 $; \

$(ii)$ \ \   $ \theta $ is an
epimorphism provided $\Gamma $ is not $K_4 $, $K_4 $ with one edge deleted, or $K_4 $ with two adjacent
edges deleted.
\end{theorem}
In the sequel, for an application of our method in proving the previous results,
 we show that the line graphs of some classes
 of odd graphs are vertex-transitive non Cayley graphs.

The graph $\Gamma$ is called a
 distance-transitive graph
if for all
vertices $u, v, x, y,$ of $\Gamma$ such that $ d(u,v) =d(x,y)$, there exists
 some $ g \in Aut(\Gamma)$ satisfying $g(u) = x $ and $ g(v) = y$. It is well known
 that every odd graph is a distance-transitive graph [2].
The action of $Aut(\Gamma)$ on
$V(\Gamma)$ induces an action on $E(\Gamma)$, by the rule
$\beta\{x,y\}=\{\beta(x),\beta(y)\}$,  $\beta\in Aut(\Gamma)$, and
$\Gamma$ is called edge-transitive if this action is
transitive.
It is clear that a distance-transitive graph is an edge-transitive graph, and hence an
odd graph is an edge-transitive graph. If a graph $\Gamma$ is
 edge-transitive, then its line graph is vertex transitive. Thus the line graph
 of an odd graph is a vertex-transitive graph.

\begin{theorem}
Let $ k > 4 $ be an even integer. If the integer ${2k+1}\choose{k}$ is a multiple of $ 4 $, then
   $ L(O_{k+1}) $, the line graph of the odd graph $O_{k+1} $  is a vertex-transitive non Cayley graph. \

\

\end{theorem}

\begin{proof}
We let $ \Gamma = L(O_{k+1}) $. Then, each vertex of $\Gamma $ is of the form $\{v,w  \} $, where
$v,w$ are $k$-subsets of $[n]= \{1,2,...,n \}$ and $v \cap w = \varnothing$. Since each vertex of
$O_{k+1}$ is of degree $k+1$, then the order of $\Gamma$
is  $\frac{k+1}{2}$${2k+1}\choose{k}$  and hence $ \Gamma$ is of
even order, because by assumption, ${2k+1}\choose{k}$ is a multiple of 4. Note that
by Theorem 2.9. each automorphism of the graph $\Gamma$ is of the form $f_{\theta}$, $\theta \in Sym([n])$, where
$f_{\theta}(\{v,w \})=\{\theta(v),\theta(w)\}$ for every vertex $\{v,w  \} $ of $\Gamma$. \

Now, On the contrary, assume that $\Gamma$ is a Cayley graph. Then the automorphism group of
the graph $\Gamma$
has a subgroup $R$ such that $R$ acts regularly on the vertex set of $\Gamma$ and hence $R$ is of order
$\frac{k+1}{2}$${2k+1}\choose{k}$. Therefore,   $R$ has an
element $f_{\theta}$ of order 2, where $\theta \in Sym([n])$, and hence $\theta$ is of order 2.
Thus, by   Remark 2.2. there are $k$-subsets $v,w$ of $[n]$ such that
$ v \cap w =\varnothing $ and $\theta (v)=v, \ \theta (w)=w$. Consequently, for the vertex $ \{v,w \}$ of the
graph $\Gamma $, we have,   \

\

\centerline{$ f_{\theta}(\{v,w \})= \{\theta(v), \theta(w)  \}= \{ v,w \} $    }
\

which is a contradiction.
\end{proof}
\
\begin{remark}
Let the group  $R$ be as in the above theorem. We do not know whether $R$ acts transitively   on the
family of $k$-subsets of $[n]$, hence it seems that we cannot use the methods of Godsil [8] for
proving the above theorem.

\end{remark}

\begin{lemma}

Let $ k > 4 $ be an even integer  and $k$ is not of the form $ k = 2^t  $  for some $ t \geq 2$,  then
the number  $ { 2k+1}  \choose { k } $ is a multiple of $ 4 $.
\end{lemma}

\begin{proof}

We know that; \

\

\centerline{$ { 2k+1 } \choose { k }$ = $  \frac  {(2k+1 )! }{ (k!)(k+1)! }$ = $ \frac {(k+2)...(2k)(2k+1)}{2.3....k} = 2(2k+1) \frac {(k+2)(k+3)...(2k-1)}{2.3...(k-1)}$}
\

If we  let $ t= \frac {(k+2)(k+3)...(2k-1)}{2.3...(k-1)}$, then we have\\

\centerline{$( k+1)t= \frac {(k+1)(k+2)...(2k-1)}{2.3...(k-1)  }$ = ${ 2(k-1)+1 } \choose { k-1 } $ = $\frac{(2k-1)!}{(k!)(k-1)!}$}

\

If we show that $t$ is a multiple of 2, then we conclude that $ 2(2k+1)t$= $ { 2k+1 } \choose { k }$  is a multiple of 4.
Note that by our assumption $k \neq 2^i$ for each positive integer $i$, hence $( k-1)$ is not of the form $2^i -1$ for some positive integer $i$, thus by Lemma 2.5. the integer  ${ 2(k-1)+1 } \choose { k-1 } $ is an even integer, and therefore 2 divides  the integer $ (k+1)t $. On the other hand, $k$ is an even integer, so $k+1$ is an odd integer,  and  hence 2 and $(k+1)$ are relatively prime, thus 2 divides the integer $t$.
\end{proof}
 Now, by Theorem 2.10. and Lemma 2.12. the following theorem follows.

\begin{theorem}
Let $ k > 4 $ be an even integer  such that $k$ is not of the form $ k = 2^t  $  for some integer  $ t > 2$.  Then
the line graph of the  odd graph $O_{k+1}$ is a vertex-transitive non Cayley graph, namely,  for such an
 integer $k$,   the graph
 $L(O_{k+1})$ is a VTNCG.

\end{theorem}
\

{ \bf Acknowledgements  } \

The author is thankful to the anonymous reviewers for valuable comments and suggestions.
\

\

\end{document}